\documentclass[12pt,reqno]{amsart} 
\usepackage{amssymb}
\usepackage{enumerate}
\usepackage[T1]{fontenc}
\usepackage[all]{xy}
\usepackage{mathrsfs}

\usepackage[usenames]{color}
\definecolor{darkgreen}{rgb}{0,0.7,0}
\definecolor{darkred}{rgb}{0.9,0,0}
\definecolor{darkblue}{rgb}{0,0,0.9}

\newlength{\shorter}
\newcommand{\mynote}[1]{{\noindent\color{darkred}\textbf{[#1]}}}

\newcommand{\lie}[3]{\def\test{#2}\def\tst{G}\ifx\test\tst{{}^{#1}#2_{#3}}
\else{{}^{#1}\!#2_{#3}}\fi}

\newcommand{\4}[1]{\widebar{#1}}
\newcommand{\5}[1]{\widehat{#1}}
\newcommand{\9}[1]{{}^{#1}\!}   %% left conjugation

\let\oldcirc=\circ
\renewcommand{\circ}{\mathchoice
    {\mathbin{\scriptstyle\oldcirc}}{\mathbin{\scriptstyle\oldcirc}}
    {\mathbin{\scriptscriptstyle\oldcirc}}
    {\mathbin{\scriptscriptstyle\oldcirc}}}

\newlength{\upto}\newlength{\dnto}
\newcommand{\I}[2]{\addtolength{\upto}{#1pt}\addtolength{\dnto}{#2pt}%
{\vrule height\upto depth\dnto width 0pt}}

\SelectTips{cm}{10} \UseTips   %% use computer modern tips in xy

\numberwithin{equation}{section}
\def\theequation{\arabic{equation}}

\let\endpf=\endproof
\renewcommand{\endproof}{\endpf\setcounter{equation}{0}}

\mathcode`\:="003A
\mathchardef\cdot="0201

\def\beq#1\eeq{\begin{equation*}#1\end{equation*}}
\def\beqq#1\eeqq{\begin{equation}#1\end{equation}}

\renewcommand{\:}{\colon}   %% as in f:X-->Y

\newcommand{\widebar}[1]{\overset{\mskip2mu\hrulefill\mskip2mu}{#1}
		\vphantom{#1}}

\newcommand{\longline}{\bigskip\centerline{\hbox to 5cm{\hrulefill}}\bigskip}

\newcommand{\mxfourb}[8]{#1&#2&#3&#4\\#5&#6&#7&#8\end{smallmatrix}\right)}
\newcommand{\mxfoura}[8]{\left(\begin{smallmatrix}#1&#2&#3&#4\\#5&#6&#7&#8\\}
\newcommand{\Mxfourb}[8]{#1&#2&#3&#4\\#5&#6&#7&#8\end{pmatrix}}
\newcommand{\Mxfoura}[8]{\begin{pmatrix}#1&#2&#3&#4\\#5&#6&#7&#8\\}
\def\trp[#1,#2,#3]{[\hskip-1.5pt[#1,#2,#3]\hskip-1.5pt]}

\DeclareMathAlphabet\EuR{U}{eur}{m}{n}
\SetMathAlphabet\EuR{bold}{U}{eur}{b}{n}

\newcommand{\higherlim}[2]{\displaystyle\setbox1=\hbox{\rm lim}
	\setbox2=\hbox to \wd1{\leftarrowfill} \ht2=0pt \dp2=-1pt
	\setbox3=\hbox{$\scriptstyle{#1}$}
	\def\test{#1}\ifx\test\empty
	\mathop{\mathop{\vtop{\baselineskip=5pt\box1\box2}}}\nolimits^{#2}
	\else
	\ifdim\wd1<\wd3
	\mathop{\hphantom{^{#2}}\vtop{\baselineskip=5pt\box1\box2}^{#2}}_{#1}
	\else
	\mathop{\mathop{\vtop{\baselineskip=5pt\box1\box2}}_{#1}}%
	\nolimits^{#2}
	\fi\fi}
\newcommand{\higherlimm}[2]{\setbox1=\hbox{\rm lim}
	\setbox2=\hbox to \wd1{\leftarrowfill} \ht2=0pt \dp2=-1pt
	\mathop{\mathop{\vtop{\baselineskip=5pt\box1\box2}}}\limits_{#1}
	\nolimits^{#2}}

\newcounter{let} 
\loop\stepcounter{let}
\expandafter\edef\csname cal\alph{let}\endcsname%
{\noexpand\mathcal{\Alph{let}}}
\ifnum\thelet<26\repeat

\newcommand{\tdef}[2][]{\expandafter\newcommand\csname#2\endcsname%
{#1\textup{#2}}}
\tdef{Iso}   \tdef{Aut}    \tdef{Out}    \tdef{Inn}    \tdef{Hom}
\tdef{End}   \tdef{Inj}    \tdef{map}    \tdef{Ker}    \tdef{Ob}
\tdef{Mor}   \tdef{Res}    \tdef{Spin}   \tdef{Id}     \tdef{Fr}
\tdef{rk}    \tdef{conj}   \tdef{incl}   \tdef{proj}   \tdef{diag} 
\tdef{trf}   \tdef{Sol}    \tdef{He}     \tdef{cj}     \tdef{Irr}
\tdef{free}  \tdef{supp}   \tdef{Ind}    \tdef{Ly}     \tdef{HS}
\tdef{Syl}
\tdef[_]{typ}   \tdef[^]{op}    \tdef[^]{ab}  \tdef{Tr}

\newcommand{\fdef}[1]{\expandafter\newcommand\csname#1\endcsname%
{\mathfrak{#1}}}
\fdef{X}  \fdef{red}  \fdef{foc}  \fdef{hyp}

\newcommand{\bbdef}[1]{\expandafter\newcommand% 
\csname#1\endcsname{\mathbb{#1}}}
\bbdef{C} \bbdef{F} \bbdef{R} \bbdef{Z} \bbdef{G} \bbdef{N}  \bbdef{B}

\newcommand{\itdef}[1]{\expandafter\newcommand\csname#1\endcsname%
{\textit{#1}}}
\itdef{PSL}  \itdef{PSU}  \itdef{SL}  \itdef{SU}  \itdef{GL} \itdef{GU}
\itdef{Sp}   \itdef{PSp}  \itdef{SO}  \itdef{SD}  \itdef{UT} \itdef{PGL}
\itdef{Sz}

\newcommand{\gen}[1]{\langle{#1}\rangle}
\newcommand{\Gen}[1]{\bigl\langle{#1}\bigr\rangle}

\let\nsg=\normal

\newcommand{\syl}[2]{\textup{Syl}_{#1}(#2)}
\newcommand{\sylp}[1]{\syl{p}{#1}}
\newcommand{\autf}{\Aut_{\calf}}

\newcommand{\outf}{\Out_{\calf}}

\newcommand{\homf}{\Hom_{\calf}}
\newcommand{\isof}{\Iso_{\calf}}
\newcommand{\sminus}{\smallsetminus}

\let\til=\widetilde

\let\too=\longrightarrow

\newcommand{\longleft}[1]{\;{\leftarrow%
\count255=0 \loop \mathrel{\mkern-6mu}%
    \relbar\advance\count255 by1\ifnum\count255<#1\repeat}\;}
\newcommand{\longright}[1]{\;{\count255=0 \loop \relbar\mathrel{\mkern-6mu}%
    \advance\count255 by1\ifnum\count255<#1\repeat\rightarrow}\;}
\newcommand{\Right}[2]{\overset{#2}{\longright#1}}
\newcommand{\RIGHT}[3]{\mathrel{\mathop{\kern0pt\longright#1}
	\limits^{#2}_{#3}}}

\newcommand{\LEFT}[3]{\mathrel{\mathop{\kern0pt\longleft#1}\limits^{#2}_{#3}}
}
\newcommand{\longleftright}[1]{\;{\leftarrow\mathrel{\mkern-6mu}%
    \count255=0\loop\relbar\mathrel{\mkern-6mu}% 
    \advance\count255 by1\ifnum\count255<#1\repeat\rightarrow}\;} 
 
\newcommand{\onto}[1]{\;{\count255=0 \loop \relbar\joinrel
    \advance\count255 by1
    \ifnum\count255<#1 \repeat \twoheadrightarrow}\;}

\newcommand{\RLEFT}[3]{\mathrel{%
   \mathop{\vcenter{\baselineskip=0pt\hbox{$\kern0pt\longright#1$}%
   \hbox{$\kern0pt\longleft#1$}}}\limits^{#2}_{#3}}}
\newcommand{\RRIGHT}[3]{\mathrel{%
   \mathop{\vcenter{\baselineskip=0pt\hbox{$\kern0pt\longright#1$}%
   \hbox{$\kern0pt\longright#1$}}}\limits^{#2}_{#3}}}

\numberwithin{table}{section}

\renewenvironment{enumerate}[1][]
{\begin{enumerat}[#1]\setlength{\itemsep}{6pt}\setlength{\topsep}{20pt}
\setlength{\partopsep}{20pt}\setlength{\parsep}{20pt}}
{\end{enumerat}}

\renewenvironment{itemize}
{\begin{itemiz}\setlength{\itemsep}{6pt}\setlength{\itemindent}{-20pt}}
{\end{itemiz}}

\newenvironment{enuma}{\begin{enumerate}[{\rm(a) }]}{\end{enumerate}}
\newenvironment{enumi}{\begin{enumerate}[{\rm(i) }]}{\end{enumerate}}

\newtheorem{Thm}{Theorem}[section]
\newtheorem{Prop}[Thm]{Proposition}
\newtheorem{Cor}[Thm]{Corollary}
\newtheorem{Lem}[Thm]{Lemma}

\newtheorem{Conj}[Thm]{Conjecture}

\theoremstyle{definition}

\newtheorem{Defi}[Thm]{Definition}

\theoremstyle{remark}

\title{Reductions to simple fusion systems}

\author{Bob Oliver}
\address{Université Paris 13, Sorbonne Paris Cité, LAGA, UMR 7539 du CNRS, 
99, Av. J.-B. Clément, 93430 Villetaneuse, France.}
\email{bobol@math.univ-paris13.fr}
\thanks{B. Oliver is partially supported by UMR 7539 of the CNRS}

\subjclass[2000]{Primary 20E25. Secondary 20D20, 20D05, 20D25, 20D45}
\keywords{Fusion systems, Sylow subgroups, finite simple groups, 
generalized Fitting subgroup, $p$-solvable groups}

\begin{document}

\begin{abstract} We prove that if $\cale\nsg\calf$ are saturated fusion 
systems over $p$-groups $T\nsg S$, such that $C_S(\cale)\le T$, and either 
$\autf(T)/\Aut_\cale(T)$ or $\Out(\cale)$ is $p$-solvable, then $\calf$ can 
be ``reduced'' to $\cale$ by alternately taking normal subsystems of 
$p$-power index or of index prime to $p$. In particular, this is the case 
whenever $\cale$ is simple and ``tamely realized'' by a known simple group 
$K$. This answers a question posed by Michael Aschbacher, and is useful 
when analyzing involution centralizers in simple fusion systems, in 
connection with his program for reproving parts of the classification of 
finite simple groups by classifying certain 2-fusion systems. 
\end{abstract}

\maketitle

When $p$ is a prime and $S$ is a finite $p$-group, a \emph{saturated fusion 
system} over $S$ is a category whose objects are the subgroups of $S$, 
whose morphisms are injective group homomorphisms between the subgroups, 
and which satisfies a certain list of axioms motivated by the Sylow 
theorems for finite groups (Definition \ref{d:sfs}). For example, when $G$ 
is a finite group and $S\in\sylp{G}$, the \emph{$p$-fusion system} of 
$G$ is the category $\calf_S(G)$ whose objects are the subgroups of $S$, 
and where for each $P,Q\le S$, $\Hom_{\calf_S(G)}(P,Q)$ is the set of those 
homomorphisms induced by conjugation in $G$.

Normal fusion subsystems of a saturated fusion system are defined by 
analogy with normal subgroups of a group (Definition \ref{d:E<|F}). Among 
the normal subsystems, we look at two particular classes: those of index 
prime to $p$ (defined over the same Sylow subgroup), and those of $p$-power 
index (see the discussions before and after Lemma \ref{l:O^p(F)}). A 
natural question arises: when $\cale\nsg\calf$, under what conditions can 
$\calf$ be ``reduced'' to $\cale$ via a sequence of steps, where one 
alternates taking normal subsystems of $p$-power index and normal 
subsystems of index prime to $p$? 

Our main theorem (Theorem \ref{solv}) says that if $\cale\nsg\calf$ are 
saturated fusion systems over $p$-groups $T\nsg S$, such that 
$C_S(\cale)\le T$, and either $\autf(T)/\Aut_\cale(T)$ or $\Out(\cale)$ is 
$p$-solvable, then $\calf$ can be reduced to $\cale$ in the above sense. In 
particular, if $\cale$ is the fusion system of a known finite simple group 
$K$, and is ``tamely realized'' by $K$ in the sense that $\Out(K)$ surjects 
onto $\Out(\cale)$ (see Section 2), then $\Out(\cale)$ is solvable 
since $\Out(K)$ is solvable by the Schreier conjecture, and hence $\calf$ 
reduces to $\cale$. 

This paper was motivated by a question posed by Michael Aschbacher. The 
above situation arises frequently when analyzing centralizers of 
involutions in simple fusion systems. If $\calf$ is the centralizer of 
an involution and $\cale=F^*(\calf)$ denotes the generalized Fitting 
subsystem of $\calf$ (see \cite[Chapter 9]{A-gfit}), then the hypothesis 
$C_S(\cale)\le T$ always holds, and $\Out(\cale)$ is solvable by Schreier's 
conjecture if $\cale/Z(\cale)$ is tamely realized by a known simple group 
$K$. Hence in this situation, Theorem \ref{solv} together with results in 
\cite{AOV1} imply that $\calf$ itself is realized by a certain extension of 
$K$. (See Corollary \ref{c:F*(F)<|F} for a slightly more general situation 
where this applies.)

Throughout the paper, all $p$-groups are assumed to be finite. Composition 
is always taken from right to left.

The author would like to thank the Universitat Aut\`onoma de Barcelona for 
its hospitality while much of this paper was being written.

\section{Normal subsystems with $p$-solvable quotient}

Since the details of the definition of a saturated fusion system play an 
important role in the proofs here, we begin by recalling some of this 
terminology. When $G$ is a group and $P,Q\le G$ are subgroups, $\Inj(P,Q)$ 
denotes the set of injective homomorphisms $P\too Q$, and 
$\Hom_G(P,Q)$ is the set of homomorphisms of the form $(x\mapsto gxg^{-1})$ 
for $g\in G$. A \emph{fusion system} over a $p$-group $S$ is 
a category $\calf$ where $\Ob(\calf)$ is the set of subgroups of $S$, and 
where for each $P,Q\le S$, 
\begin{enumerate}[\rm(i) ]
\item $\Hom_S(P,Q)\subseteq\homf(P,Q) \subseteq\Inj(P,Q)$, and 
\item $\varphi\in\homf(P,Q)$ implies $\varphi^{-1}\in\homf(\varphi(P),P)$. 
\end{enumerate}
By analogy with the terminology for groups, we say that $P,Q\le S$ are 
\emph{$\calf$-conjugate} if they are isomorphic in $\calf$, and let 
$P^\calf$ denote the set of subgroups $\calf$-conjugate to $P$.

\begin{Defi}[{\cite{RS}}] \label{d:sfs}
Let $\calf$ be a fusion system over a $p$-group $S$.
\begin{enuma} 
\item A subgroup $P\le S$ is \emph{fully automized in $\calf$} if 
$\Aut_S(P)\in\sylp{\autf(P)}$.

\item A subgroup $P\le S$ is \emph{receptive in $\calf$} if for each $Q\in 
P^\calf$ and each $\varphi\in\isof(Q,P)$, $\varphi$ extends to a 
homomorphism $\4\varphi\in\homf(N_\varphi,S)$, where
	\[ N_\varphi = \bigl\{ g\in N_S(Q) \,\big|\, \varphi 
	c_g\varphi^{-1}\in\Aut_S(P) \bigr\}. \]

\item The fusion system $\calf$ is \emph{saturated} if each 
$\calf$-conjugacy class of subgroups of $S$ contains at least one member 
which is fully automized and receptive in $\calf$. 

\item When $\calh$ is a set of subgroups of $\calf$ closed under 
$\calf$-conjugacy, we say that $\calf$ is \emph{$\calh$-saturated} if each 
member of $\calh$ is $\calf$-conjugate to a subgroup which is fully 
automized and receptive in $\calf$. We say that $\calf$ is 
\emph{$\calh$-generated} if each morphism in $\calf$ is a composite of 
restrictions of morphisms between members of $\calh$. 

\item When $X$ is a set of injective homomorphisms between subgroups of 
$S$, $\gen{X}$ denotes the fusion system over $S$ generated by $X$: 
the smallest fusion system which contains $X$. When $\calf_0$ 
is a category whose objects are subgroups of $S$ and whose morphisms are 
injective homomorphisms, we write $\gen{\calf_0}=\gen{\Mor(\calf_0)}$.

\end{enuma}
\end{Defi}

We next recall some of the terminology for subgroups in a fusion system. 
When $\calf$ is a fusion system over $S$ and $P\le S$, we write 
$\outf(P)=\autf(P)/\Inn(P)$.

\begin{Defi} \label{d:subgroups}
Let $\calf$ be a fusion system over a $p$-group $S$, and let $P\le S$ be a 
subgroup. Then 
\begin{enuma} 
\item $P$ is \emph{fully centralized (fully normalized) in} $\calf$ if 
$|C_S(P)|\ge|C_S(Q)|$ ($|N_S(P)|\ge|N_S(Q)|$) for each $Q\in P^\calf$;

\item $P$ is \emph{$\calf$-centric} if $C_S(Q)\le Q$ for each $Q\le P$;

\item $P$ is \emph{$\calf$-radical} if $O_p(\outf(P))=1$; and 

\item $P$ is \emph{strongly closed in $\calf$} if $\varphi(P_0)\le P$ for each 
$P_0\le P$ and each $\varphi\in\homf(P_0,S)$.
\end{enuma}
We also let $\calf^{cr}\subseteq\calf^c$ denote the sets of subgroups of 
$S$ which are $\calf$-centric and $\calf$-radical, or $\calf$-centric, 
respectively.
\end{Defi}

The following lemma describes the relations between some of these 
conditions on subgroups. Point (a) and (b) are due to Roberts and 
Shpectorov \cite{RS} and are also shown in \cite[Lemma I.2.6(c)]{AKO}, 
while (c) is immediate from the definitions.

\begin{Lem} \label{f.c.=>receptive}
In a saturated fusion system $\calf$ over a $p$-group $S$, for each 
subgroup $P\le S$,
\begin{enuma} 
\item $P$ is fully centralized if and only if $P$ is receptive; 
\item $P$ is fully normalized if and only if $P$ is fully automized and 
receptive; and
\item if $P$ is $\calf$-centric, then $P$ is fully centralized and hence 
receptive.
\end{enuma}
\end{Lem}

As described in the introduction, normal fusion subsystems play a central 
role here.

\begin{Defi} \label{d:E<|F}
Let $\calf$ be a saturated fusion system over a $p$-group $S$, and let 
$\cale\le\calf$ be a saturated fusion subsystem over $T\le S$. 
The subsystem $\cale$ is \emph{normal} in $\calf$ ($\cale\nsg\calf$) if 
\begin{itemize} 
\item $T$ is strongly closed in $\calf$ (in particular, $T\nsg S$);

\item (invariance condition) each $\alpha\in\autf(T)$ is fusion preserving 
with respect to $\cale$ (i.e., extends to $(\alpha,\5\alpha)\in\Aut(\cale)$); 

\item (Frattini condition) for each $P\le T$ and each 
$\varphi\in\homf(P,T)$, there are $\alpha\in\autf(T)$ and 
$\varphi_0\in\Hom_\cale(P,T)$ such that $\varphi=\alpha\circ\varphi_0$; and 

\item (extension condition) each $\alpha\in\Aut_\cale(T)$ extends to some 
$\4\alpha\in\autf(TC_S(T))$ such that $[\4\alpha,C_S(T)]\le Z(T)$.
\end{itemize}
\end{Defi}

Finally, we will frequently need to use the following version of Alperin's 
fusion theorem for fusion systems. This is the version shown in 
\cite[Theorem A.10]{BLO2}. For a stronger version due to Puig, see, e.g., 
\cite[Theorems I.3.5--6]{AKO}.

\begin{Thm} \label{AFT}
If $\calf$ is a saturated fusion system over a $p$-group $S$, then 
	\[ \calf=\Gen{\autf(P) \,\big|\, \textup{$P\le S$ is 
	$\calf$-centric, $\calf$-radical, and fully normalized in $\calf$}}. 
	\]
Thus each morphism in $\calf$ is a composite of restrictions of 
$\calf$-automorphisms of such subgroups.
\end{Thm}

If $\cale\nsg\calf$ are saturated fusion systems, then $\cale$ has 
\emph{index prime to $p$} in $\calf$ if they are both over the same 
$p$-group $S$, and for each $P\le S$, $\Aut_\cale(P)\ge O^{p'}(\autf(P))$. 
Let $O^{p'}(\calf)\nsg\calf$ be the smallest normal subsystem of index 
prime to $p$ in $\calf$ \cite[Theorem 5.4]{BCGLO2}. We first fix some tools 
for constructing normal subsystems of this type. The following lemma is 
basically the same as Theorem I.7.7 in \cite{AKO}, but stated in a slightly 
more general setting.

\begin{Lem} \label{Ker(chi)}
Let $\calf$ be a saturated fusion system over a $p$-group $S$. Let 
$\calh\subseteq\calf^c$ be a nonempty set of $\calf$-centric subgroups of 
$S$ such that 
\begin{enumi} 
\item $\calh$ is closed under $\calf$-conjugacy and overgroups; and
\item for each $P\in\calf^c\sminus\calh$, there is $P^*\in P^\calf$ such 
that $\Out_S(P^*)\cap O_p(\outf(P^*))\ne1$.
\end{enumi}
Let $\calf^*\subseteq\calf$ be the full subcategory with object set 
$\calh$, let $\Delta$ be a finite group of order prime to $p$, and let 
$\chi\: \Mor(\calf^*) \Right2{} \Delta$ be such that 
\begin{enumi}\setcounter{enumi}{2}
\item $\chi(\incl_P^S)=1$ for all $P\in\calh$,
\item $\chi(\psi\varphi)=\chi(\psi)\chi(\varphi)$ for each composable pair of 
morphisms $\psi$ and $\varphi$ in $\calf^*$, and 
\item $\chi(\autf(S))=\Delta$. 
\end{enumi}
Let $\calf_0^*\subseteq\calf^*$ be the subcategory with the same objects, 
and with $\Mor(\calf_0^*)=\chi^{-1}(1)$. Set $\calf_0=\gen{\calf_0^*}$: the 
fusion system over $S$ generated by $\calf_0^*$. Then $\calf_0\nsg\calf$ is 
a normal saturated fusion subsystem, $O^{p'}(\calf)\le\calf_0\le\calf$, and 
$\Aut_{\calf_0}(S)=\Ker(\chi|_{\autf(S)})$.
\end{Lem}

\begin{proof} By (iv), $\calf_0^*$ is a subcategory of $\calf^*$. 
Since $|\Delta|$ is prime to $p$, $\chi(\Inn(S))=1$, and 
hence $\Inn(S)\le\Aut_{\calf_0^*}(S)$. Also, by (iii) and (iv), 
	\beqq \textup{$\chi(\varphi')=\chi(\varphi)$ when 
	$\varphi'\in\Mor(\calf^*)$ is any 
	restriction of $\varphi\in\Mor(\calf^*)$.} \label{e:same_chi} \eeqq
Since $\calh$ is closed under overgroups by (i), we thus have 
$\Hom_{\calf_0}(P,Q)=\Hom_{\calf_0^*}(P,Q)$ for all $P,Q\in\calh$.

We next show the following:
\begin{enumerate}[(1) ] \setcounter{enumi}{\theequation} 

\item \label{e:Fr.cond.}
For each $P,Q\in\calh$ and $\varphi\in\Hom_{\calf^*}(P,Q)$, there are 
morphisms $\alpha\in\autf(S)$ and 
$\varphi_0\in\Hom_{\calf_0^*}(P,Q^*)$, where $Q^*=\alpha^{-1}(Q)$, such that 
$\varphi=\alpha|_{Q^*}\circ\varphi_0$.

\item \label{e:P^F0}
If $P,Q\in\calh$ and $Q\in P^\calf$, then there is 
$\alpha\in\autf(S)$ such that $\alpha(Q)\in P^{\calf_0}$.

\item \label{e:F0-f.aut.}
If $P\in\calh$ is fully automized in $\calf$, then it is fully 
automized in $\calf_0$. 

\item \label{e:F0-recept.}
If $P\in\calh$ is receptive in $\calf$, then it is receptive in 
$\calf_0$. 

\end{enumerate} \setcounter{equation}{\theenumi}
Point \eqref{e:Fr.cond.} follows from (v) (and \eqref{e:same_chi}): choose 
$\alpha$ such that $\chi(\alpha)=\chi(\varphi)$, and set 
$\varphi_0=\alpha^{-1}\varphi\in\Mor(\calf_0^*)$. Point \eqref{e:P^F0} 
follows immediately from \eqref{e:Fr.cond.}. 
Point \eqref{e:F0-f.aut.} holds since if $\Aut_S(P)\in\sylp{\autf(P)}$, 
then $\Aut_S(P)$ is also a Sylow $p$-subgroup in $\Aut_{\calf_0}(P)$.

If $P\in\calh$ is receptive in $\calf$, then each 
$\varphi\in\Hom_{\calf_0}(Q,S)$ with $\varphi(Q)=P$ extends to some 
$\4\varphi\in\homf(N_\varphi,S)$, where $N_\varphi=\{x\in 
N_S(Q)\,|\,\varphi c_x\varphi^{-1}\in\Aut_S(P)\}$. Then $N_\varphi\in\calh$ 
by (i) and $\chi(\4\varphi)=\chi(\varphi)=1$ by \eqref{e:same_chi}, so 
$\4\varphi\in\Hom_{\calf_0}(N_\varphi,S)$. Thus $P$ is receptive in 
$\calf_0$ in this case, and this proves \eqref{e:F0-recept.}. 

Since $\calf$ is saturated, for each $P\in\calh$, there is $Q\in P^\calf$ 
which is fully automized and receptive in $\calf$. By \eqref{e:P^F0}, there 
is $\alpha\in\autf(S)$ such that $\alpha(Q)\in P^{\calf_0}$. Then 
$\alpha(Q)$ is also fully automized and receptive in $\calf$, hence in 
$\calf_0$ by \eqref{e:F0-f.aut.} and \eqref{e:F0-recept.}. Thus $\calf_0$ 
is $\calh$-saturated, and it is $\calf_0$-generated by definition. So by 
(ii) and \cite[Theorem 2.2]{BCGLO1}, $\calf_0$ is saturated. 

We claim that 
	\beqq \alpha\in\autf(S),~ \varphi\in\Mor(\calf_0) 
	\quad\implies\quad \alpha\varphi\alpha^{-1}\in\Mor(\calf_0). 
	\label{e:invar.cond.} \eeqq
Here, $\alpha\varphi\alpha^{-1}$ means composition on each side with the 
appropriate restriction of $\alpha$ or $\alpha^{-1}$. This holds for 
$\varphi\in\Mor(\calf_0^*)$ by definition (and \eqref{e:same_chi}), and 
hence holds for all composites of restrictions of such morphisms. Since 
$\calf_0=\gen{\calf_0^*}$, \eqref{e:invar.cond.} holds for all 
$\varphi\in\Mor(\calf_0)$.

We next check that 
	\beqq 
	\forall~ \varphi\in\Mor(\calf),\quad \exists~
	\varphi_0\in\Mor(\calf_0),~ \alpha\in\autf(S) \quad\textup{such 
	that}\quad \varphi=\alpha\varphi_0. \label{e:Frat.cond.} \eeqq
Since $\calh\supseteq\calf^{cr}$ by (ii), $\calf=\gen{\calf^*}$ by 
Alperin's fusion theorem. Hence $\calf=\gen{\calf_0,\autf(S)}$ by 
\eqref{e:Fr.cond.}. By \cite[Lemma 3.4.c]{BCGLO2} and 
\eqref{e:invar.cond.}, this suffices to prove \eqref{e:Frat.cond.}.

Since the extension condition for normality holds trivially in this 
case, this proves that $\calf_0\nsg\calf$ (see \cite[Definition 
I.6.1]{AKO}). Hence $\calf_0\ge O^{p'}(\calf)$, since $\calf$ and $\calf_0$ 
are both saturated fusion systems over $S$ (see \cite[Lemma 1.26]{AOV1}). 
\end{proof}

The next lemma is needed to check that point (ii) holds when we apply Lemma 
\ref{Ker(chi)}.

\begin{Lem} \label{L<|L*}
Let $\cale\nsg\calf$ be saturated fusion systems over $T\nsg S$. Then for 
each $P\in\calf^c$ such that $P\cap T\notin\cale^c$, there is $P^*\in 
P^\calf$ such that $\Out_S(P^*)\cap O_p(\outf(P^*))\ne1$. In particular, 
$P\notin\calf^{cr}$.
\end{Lem}

\begin{proof} By \cite[Lemma I.2.6.c]{AKO}, there is 
$\varphi\in\homf(N_S(P\cap T),S)$ such that $\varphi(P\cap T)$ is fully 
normalized in $\calf$. Set $P^*=\varphi(P)$; then $\varphi(P\cap T)=P^*\cap T$ 
since $T$ is strongly closed in $\calf$, and hence $P^*\cap T$ is fully 
normalized in $\calf$. Since $P\cap T\notin\cale^c$, there is $Q\in(P\cap 
T)^\cale\subseteq(P^*\cap T)^\calf$ such that $C_T(Q)\nleq Q$, each 
$\psi\in\isof(Q,P^*\cap T)$ extends to $QC_S(Q)$, and hence $C_T(P^*\cap 
T)\nleq P^*$. 

Thus $P^*C_T(P^*\cap T)>P^*$, so 
$N_{P^*C_T(P^*\cap T)}(P^*)>P^*$, and there is $x\in N_T(P^*)\sminus P^*$ 
such that $[x,P^*\cap T]=1$. Conjugation by $x$ induces the identity on 
$P^*\cap T$ and on $P^*/(P^*\cap T)$, so $c_x\in O_p(\autf(P^*))$. Also, 
$c_x\notin\Inn(P^*)$ since $P^*\in\calf^c$, so $1\ne[c_x]\in\Out_S(P^*)\cap 
O_p(\outf(P^*))$. 
\end{proof}

The next proposition provides a more explicit way to construct proper 
normal subsystems of index prime to $p$. Note that the existence of a 
normal subsystem over the strongly closed subgroup $T$ is crucial, as is 
clearly seen by considering the case where $T=S$. In fact, the proposition 
is rather trivial when $T=S$ or $T=1$, and is useful only when $1\ne T<S$.

\begin{Prop} \label{O^p'(F)}
Let $\cale\nsg\calf$ be saturated fusion systems over $p$-groups $T\nsg S$. Let 
$\chi_0\:\autf(T)\too\Delta$ be a surjective homomorphism, for some 
$\Delta\ne1$ of order prime to $p$, such that 
$\Aut_\cale(T)\le\Ker(\chi_0)$. Then there is a unique proper 
normal subsystem $\calf_0\nsg\calf$ over $S$ such that 
	\beqq \Aut_{\calf_0}(S)=\bigl\{\alpha\in\autf(S)\,\big|\,
	\alpha|_T\in\Ker(\chi_0) \bigr\} \label{e:AutF0(S)} \eeqq
and $\calf_0\ge\cale$. In particular, $O^{p'}(\calf)\le\calf_0<\calf$.
\end{Prop}

\begin{proof} The uniqueness of $\calf_0$ follows from \eqref{e:AutF0(S)} 
and \cite[Theorem 5.4]{BCGLO2}: a saturated fusion subsystem of index prime 
to $p$ in $\calf$ is uniquely determined by the automizer of $S$.

Let $\calf|_{\cale^c}\subseteq\calf$ be the full subcategory 
with objects in $\cale^c$. We first show that there is a map 
	\[ \chi\: \Mor(\calf|_{\cale^c}) \Right4{} \Delta \]
which extends $\chi_0$, which sends composites to products, and which sends 
$\Mor(\cale^c)$ to the identity. By the Frattini condition for a normal 
fusion subsystem \cite[Definition I.6.1]{AKO}, for each $P,Q\in\cale^c$ and 
each $\varphi\in\homf(P,Q)$, there are $\alpha\in\autf(T)$ and 
$\varphi_0\in\Hom_\cale(P,Q_1)$, where $Q_1=\alpha^{-1}(Q)$, such that 
$\varphi=\alpha|_{Q_1}\circ\varphi_0$. In this situation, the conditions 
imposed on $\chi$ imply that $\chi(\varphi)=\chi(\alpha)=\chi_0(\alpha)$. 
It remains to prove that this is independent of the choice of decomposition 
of $\varphi$, and that it sends composites to products. 

To see that $\chi$ sends composites to products when it is uniquely 
defined, fix a composable pair of morphisms $\psi,\varphi\in\Mor(\calf)$: 
a pair such that $\psi\circ\varphi$ is 
defined. Assume $\varphi=\alpha\varphi_0$ and $\psi=\beta\psi_0$ (after 
appropriate restrictions of $\alpha$ and $\beta$), where 
$\varphi_0,\psi_0\in\Mor(\cale)$ and $\alpha,\beta\in\autf(T)$. Thus 
$\chi(\varphi)=\chi_0(\alpha)$ and $\chi(\psi)=\chi_0(\beta)$. Also, 
	\[ \psi\circ\varphi=\beta\psi_0\alpha\varphi_0
	=(\beta\alpha)(\alpha^{-1}\psi_0\alpha)\varphi_0 \]
where $\alpha^{-1}\psi_0\alpha\in\Mor(\cale)$ since 
$\cale\nsg\calf$. So $\chi(\psi\circ\varphi)=\chi_0(\beta\alpha)$.

Again fix $P,Q\in\cale^c$ and $\varphi\in\homf(P,Q)$. Let 
$\varphi=\alpha|_{Q_1}\circ\varphi_0=\beta|_{Q_2}\circ\psi_0$ be two 
decompositions, where $\alpha,\beta\in\autf(T)$, $Q_1=\alpha^{-1}(Q)$, 
$Q_2=\beta^{-1}(Q)$, and $\varphi_0,\psi_0\in\Mor(\cale)$. If $P=Q=T$, then 
all of these morphisms lie in $\autf(T)$, and 
$\chi_0(\alpha)=\chi_0(\beta)$ since $\chi_0(\Aut_\cale(T))=1$. So assume 
$P<T$, and also assume inductively that $\chi$ is uniquely defined on all 
morphisms between subgroups in $\cale^c$ strictly larger than $P$. Then 
$(\beta^{-1}\alpha|_{Q_1})\varphi_0=\psi_0\in\Hom_\cale(P,Q_2)$. Since 
$\beta^{-1}\alpha|_{Q_1}\in\Hom_\cale(Q_1,Q_2)$ extends in $\calf$ to $T$, 
it also extends in $\cale$ to $N_T(Q_1)>Q_1$ (recall that all 
$\cale$-centric subgroups are receptive by Lemma \ref{f.c.=>receptive}(c)). 
Let $\gamma\in\Hom_\cale(N_T(Q_1),T)$ be such that 
$\beta^{-1}\alpha|_{Q_1}=\gamma|_{Q_1}$. Since $Q_1\in\cale^c$, 
$\beta^{-1}\alpha(N_T(Q_1))=\gamma(N_T(Q_1))$, so 
$\alpha^{-1}\beta\gamma\in\autf(N_T(Q_1))$. This automorphism has $p$-power 
order since it is the identity on $Q_1$ and on $N_T(Q_1)/Q_1$, and hence 
$\chi(\alpha^{-1}\beta\gamma)=1$. Since all of these homomorphisms involve 
subgroups strictly larger than $P$, 
$\chi(\alpha)^{-1}\chi(\beta)\chi(\gamma)=1$, where $\chi(\gamma)=1$ since 
$\gamma\in\Mor(\cale^c)$. So $\chi_0(\alpha)=\chi_0(\beta)$, and the two 
decompositions of $\varphi$ give the same value for $\chi(\varphi)$. 

Thus $\chi$ is uniquely defined. Set
	\[ \calh^* = \{P\in\calf^c \,|\, P\cap T\in\cale^c \}, \]
and let $\5\chi$ be the composite
	\[ \5\chi\: \Mor(\calf|_{\calh^*}) \Right4{R} \Mor(\calf|_{\cale^c}) 
	\Right4{\chi} \Delta , \]
where $R$ sends $\varphi\in\homf(P,Q)$ to $\varphi|_{P\cap T}\in 
\homf(P\cap T,Q\cap T)$. Since $T\nsg S$, $T$ is fully 
normalized, and hence is fully automized and receptive (Lemma 
\ref{f.c.=>receptive}(b)). Hence $\Aut_S(T)$ lies in 
$\sylp{\autf(T)}=\sylp{O^{p'}(\autf(T))}$, and so 
	\[ \autf(T) = N_{\autf(T)}(\Aut_S(T))\cdot O^{p'}(\autf(T)) \]
by the Frattini argument. Since $T$ is receptive, each 
$\alpha\in N_{\autf(T)}(\Aut_S(T))$ extends to some $\4\alpha\in\autf(S)$, and 
$\5\chi(\alpha)=\5\chi(\4\alpha)\5\chi(\incl_T^S)=\5\chi(\4\alpha)$. Since 
$\5\chi(O^{p'}(\autf(T)))=1$, $\5\chi(\autf(S))=\5\chi(\autf(T))=\Delta$. 

Condition (ii) in Lemma \ref{Ker(chi)} holds by Lemma \ref{L<|L*}. We 
just checked condition (v), condition (iv) holds for $\5\chi$ since it 
holds for $\chi$, and the other two are clear. So by that 
lemma, there is $\calf_0<\calf$ which is normal in $\calf$ and contains 
$O^{p'}(\calf)$, and such that $\Aut_{\calf_0}(S)$ is as required.

It remains to show that $\calf_0\ge\cale$. If $P\in\cale^c$ and $P$ is 
fully centralized in $\calf$, then each $\varphi\in\Aut_\cale(P)$ extends 
to some $\4\varphi\in\autf(PC_S(P))$, where $PC_S(P)\in\calf^c$ and 
$\5\chi(\4\varphi)=\chi(\varphi)=1$, so $\4\varphi$ and hence $\varphi$ are 
in $\Mor(\calf_0)$. If $P\in\cale^c$ is arbitary, then $\psi(P)$ is fully 
centralized in $\calf$ for some $\psi\in\homf(P,T)$, and 
$\Aut_\cale(P)=(\Aut_\cale(\varphi(P)))^\varphi$ and 
$\Aut_{\calf_0}(P)=(\Aut_{\calf_0}(\varphi(P)))^\varphi$ since $\cale$ and 
$\calf_0$ are normal in $\calf$ (see \cite[Proposition I.6.4(d)]{AKO}, 
applied with $Q=P$). Hence $\Aut_\cale(P)\le\Aut_{\calf_0}(P)$ for all 
$P\in\cale^c$, and $\cale\le\calf_0$ by Alperin's fusion theorem.
\end{proof}

\begin{Cor} 
If $\cale\nsg\calf$ are saturated fusion systems over $T\nsg S$, and 
	\[ \Aut_\cale(T)O^{p'}(\autf(T))<\autf(T), \]
then $O^{p'}(\calf)<\calf$.
\end{Cor}

We now turn to constructions of normal subsystems of $p$-power index. 
Recall first the definition of the hyperfocal subgroup $\hyp(\calf)$ for a 
saturated fusion system $\calf$ over a $p$-group $S$:
	\[ \hyp(\calf) = \Gen{[O^p(\autf(P)),P] \,\big|\, P\le S} \nsg S. \]

\begin{Lem} \label{l:O^p(F)}
Let $\calf$ be a saturated fusion system over a $p$-group $S$, and assume 
$T\nsg S$ is strongly closed in $\calf$. Then 
\begin{enuma} 

\item $\hyp(\calf/T)=T\cdot\hyp(\calf)/T$;

\item $TC_S(T)$ is strongly closed in $\calf$; and 

\item the natural isomorphism $S/TC_S(T)\cong\Out_S(T)$ extends to an 
isomorphism of fusion systems 
$\calf/TC_S(T)\cong\calf_{\Out_S(T)}(\outf(T))$.

\end{enuma}
\end{Lem}

\begin{proof} By \cite[Theorem II.5.12]{AKO}, and since $T$ is strongly 
closed in $\calf$, there is a morphism of fusion systems 
$(\Psi,\5\Psi)\:\calf\too\calf/T$ which sends $P$ to $PT/T$ and sends 
$\varphi\in\homf(P,Q)$ to the induced homomorphism between quotient groups. 

Set $\5T=TC_S(T)$ for short. 

\smallskip

\noindent\textbf{(a) } If $P\le S$, and $\varphi\in\autf(P)$ has order 
prime to $p$, then $\5\Psi(\varphi)\in\Aut_{\calf/T}(PT/T)$ also has order 
prime to $p$. Hence $T[\varphi,P]/T\le\hyp(\calf/T)$. Since $\hyp(\calf)$ 
is generated by such commutators $[\varphi,P]$ by definition, this proves 
that $T\cdot\hyp(\calf)/T\le\hyp(\calf/T)$. 

Conversely, for each $P/T\le S/T$, and each $\psi\in\Aut_{\calf/T}(P/T)$ of 
order prime to $p$, $\psi$ lifts to some $\5\psi\in\autf(P)$ by definition 
of $\calf/T$, and $[\psi,P/T]=T[\5\psi,P]/T$ where 
$[\5\psi,P]\le\hyp(\calf)$. Since $\hyp(\calf/T)$ is generated by such 
commutators $[\psi,P/T]$, this proves that $\hyp(\calf/T)\le T\cdot\hyp(\calf)/T$.

\smallskip

\noindent\textbf{(b) } Fix $P\le \5T$ and $\varphi\in\homf(P,S)$. Choose 
$\5\varphi\in\homf(PT,S)$ such that 
$\5\Psi(\5\varphi)=\5\Psi(\varphi)\in\Hom_{\calf/T}(PT/T,S/T)$. Then 
$\5\varphi(T)=T$, $PT\le \5T=TC_S(T)$, so $PT=TC_{PT}(T)$, and 
$\varphi(P)\le\5\varphi(PT)=TC_{\5\varphi(PT)}(T)\le\5T$. Thus $\5T$ is 
strongly closed.

\smallskip

\noindent\textbf{(c) } Fix $P,Q\le S$ which contain $\5T$, and 
$\varphi\in\homf(P,Q)$. Let $\5\varphi\in\Hom_{\calf/\5T}(P/\5T,Q/\5T)$ be 
the induced homomorphism, and let $[\varphi|_T]\in \outf(T)$ be the class 
of $\varphi|_T\in\autf(T)$. Then for $x\in P$ and $c_x\in\Aut_P(T)$, 
$(\varphi|_T)c_x(\varphi|_T)^{-1}=c_{\varphi(x)}\in\Aut_Q(T)$. So if we 
identify $\Out_P(T)\cong P/\5T$ and $\Out_Q(T)\cong Q/\5T$, then 
$\5\varphi$ is conjugation by $[\varphi|_T]$, and hence a morphism in 
$\calf_{\Out_S(T)}(\outf(T))$.

Conversely, if conjugation by the class of $\psi\in\autf(T)$ sends 
$\Out_P(T)$ into $\Out_Q(T)$ for $P,Q\ge \5T$, then $\psi$ extends to some 
$\psi^*\in\homf(P,Q)$ since $T\nsg S$ is receptive (Lemma 
\ref{f.c.=>receptive}(b)), and hence 
$c_{[\psi]}\in\Hom_{\outf(T)}(\Out_P(T),\Out_Q(T))$ is identified with 
$\5\psi^*\in\Hom_{\calf/\5T}(P/\5T,Q/\5T)$. 
\end{proof}

In \cite[\S\,3]{BCGLO2}, a fusion subsystem $\calf_0\le\calf$ over $U\le S$ 
is defined to have \emph{$p$-power index} if $U\ge\hyp(\calf)$ and 
$\Aut_{\calf_0}(P)\ge O^p(\autf(P))$ for each $P\le U$. By \cite[Theorem 
I.7.4]{AKO}, if $\calf$ is saturated, then for each $U\le S$ containing 
$\hyp(\calf)$, there is a unique saturated fusion subsystem 
$\calf_U\le\calf$ over $U$ of $p$-power index in $\calf$, and 
$\calf_U\nsg\calf$ if $U\nsg S$. 

\begin{Prop} \label{O^p(F)}
If $\calf$ is a saturated fusion system over a $p$-group $S$, and 
$T\nsg S$ is strongly closed in $\calf$, then 
	\[ \hyp(\calf) \le \{x\in S\,|\, c_x\in O^p(\autf(T))\Inn(T)\}. \]
If, in addition, $\cale\nsg\calf$ and $\calf_0\nsg\calf$ are normal 
subsystems over $T$ and $U$, respectively, where $U\ge T\cdot\hyp(\calf)$ and 
$\calf_0$ has $p$-power index in $\calf$, then $\cale\le\calf_0$.
\end{Prop}

\begin{proof} Set $\5T=TC_S(T)$ and $Q=\{x\in S\,|\,c_x\in 
O^p(\autf(T))\Inn(T)\}\ge\5T$, for short, and let 
$\omega\:S/\5T\Right2{\cong}\Out_S(T)$ be the natural isomorphism. By Lemma 
\ref{l:O^p(F)}(b), $\5T$ is strongly closed in $\calf$, and by Lemma 
\ref{l:O^p(F)}(c), $\omega$ induces an isomorphism of fusion systems 
$\calf/\5T\cong\calf_{\Out_S(T)}(\outf(T))$. Puig's hyperfocal theorem for 
groups now implies that 
	\begin{align*} 
	\omega(\hyp(\calf/\5T)) &= \hyp\bigl(\calf_{\Out_S(T)}(\outf(T))\bigr)
	= O^p(\outf(T))\cap\Out_S(T) \\
	&= \bigl(O^p(\autf(T))\Inn(T) \cap \Aut_S(T)\bigr) 
	\big/ \Inn(T) \\
	&= \Aut_Q(T)/\Inn(T) = \omega(Q/\5T).
	\end{align*}
Hence $\hyp(\calf/\5T)=Q/\5T$, so $Q=\5T\cdot\hyp(\calf)$ by Lemma 
\ref{l:O^p(F)}(a). 

Now let $U\nsg S$ be any normal subgroup containing $T\cdot\hyp(\calf)$, and 
assume that $\cale$ and $\calf_0$ are normal subsystems in $\calf$ over $T$ 
and $U$, respectively, where $\calf_0$ has $p$-power index in $\calf$. If 
$P\le T$ is fully normalized in $\cale$, then 
$\Aut_\cale(P)=\Aut_T(P)O^p(\Aut_\cale(P))$ since 
$\Aut_T(P)\in\sylp{\Aut_\cale(P)}$, $\Aut_T(P)\le\Aut_{\calf_0}(P)$ since 
$\calf_0$ is a fusion system over $U\ge T$, and $O^p(\Aut_\cale(P))\le 
O^p(\autf(P))\le\Aut_{\calf_0}(P)$ where the second inclusion holds since 
$\calf_0\nsg\calf$ has $p$-power index. Hence 
$\Aut_\cale(P)\le\Aut_{\calf_0}(P)$ for all such $P$, and $\cale\le\calf_0$ 
since $\cale$ is generated by such automorphisms by Alperin's fusion 
theorem. 
\end{proof}

The next lemma will be needed to show that certain subnormal systems 
are normal, when iterating inductively Propositions \ref{O^p'(F)} and 
\ref{O^p(F)}.

\begin{Lem}[{\cite[7.4]{A-gfit}}] \label{F2<|F1<|F}
Let $\calf_2\nsg\calf_1\nsg\calf$ be saturated fusion systems over 
$p$-groups $S_2\nsg S_1\nsg S$. Assume, for each $\alpha\in\autf(S_1)$, 
that $\9\alpha\calf_2=\calf_2$. Assume also that $C_S(S_2)\le S_2$. Then 
$\calf_2\nsg\calf$.
\end{Lem}

For an arbitrary saturated fusion system $\calf$, let 
$\calf^\infty\nsg\calf$ be the limit after applying $O^p(-)$ and 
$O^{p'}(-)$ until the sequence stabilizes. More precisely, set 
$\calf^\infty=\bigcap_{i=0}^\infty\calf^{(i)}$, where the sequence 
$\{\calf^{(i)}\}$ is defined by setting $\calf^{(0)}=\calf$ and 
$\calf^{(i+1)}=O^{p'}(O^p(\calf^{(i)}))$. 

\begin{Lem} \label{unique-Finfty}
Let $\calf=\calf_0>\calf_1>\cdots>\calf_m=\cale$ be any sequence of 
saturated fusion systems, each normal in $\calf$ and each of $p$-power 
index or of index prime to $p$ in the preceeding one. Then 
$\calf^\infty=\cale^\infty$.
\end{Lem}

\begin{proof}
We prove that $\cale^\infty=\calf^\infty$ by induction on $|\Mor(\calf)|$. 
In particular, it suffices to do this when $m=1$. If $\cale\nsg\calf$ has 
$p$-power index, then $O^p(\cale)=O^p(\calf)$ \cite[Theorem 
7.53(ii)]{Craven}, and $\cale^\infty=\calf^\infty$ by definition. So assume $\cale$ has index 
prime to $p$ in $\calf$. Set $T=\hyp(\calf)$. Since $\hyp(\cale)\le T$ by 
definition, there is a unique normal subsystem $\cale_T\nsg\cale$ of 
$p$-power index over $T$, and by \cite[Theorem I.7.4]{AKO}, 
$\cale_T=\gen{\Inn(T),O^p(\Aut_\cale(P))\,|\,P\le T}\le O^p(\calf)$. Also, 
$\cale_T\nsg\calf$ by Lemma \ref{F2<|F1<|F} and the uniqueness of $\cale_T$ 
(and since $T\nsg S$). Hence $\cale_T\nsg 
O^p(\calf)$, and $\cale_T$ has index prime to $p$ in $O^p(\calf)$ since 
they are both fusion systems over $T$ \cite[Lemma 1.26]{AOV1}. Hence 
	\[ \cale^\infty = (\cale_T)^\infty = 
	(O^{p'}(O^p(\calf)))^\infty = \calf^\infty, \]
where the second equality holds by the induction hypothesis.
\end{proof}

We are now ready to combine Propositions \ref{O^p'(F)} and \ref{O^p(F)} to 
get the following:

\begin{Thm} \label{t:p-solv}
Let $\cale\nsg\calf$ be saturated fusion systems over $p$-groups $T\nsg S$. 
Assume that $\autf(T)/\Aut_\cale(T)$ is $p$-solvable (equivalently, that 
$\outf(T)$ is $p$-solvable). Then there is a normal saturated subsystem 
$\calf_0\nsg\calf$ over $TC_S(T)$, such that $\calf_0\ge\cale$, 
$\calf_0\ge\calf^\infty$, $(\calf_0)^\infty=\calf^\infty$, and 
$\Aut_{\calf_0}(T)=\Aut_\cale(T)$.
\end{Thm}

\begin{proof} Choose a sequence of subgroups 
$\autf(T)=G_m>G_{m-1}>\cdots>G_0=\Aut_\cale(T)$, all normal in $G_m$, and 
such that for each $i$, $G_i/G_{i-1}$ is a $p$-group or a $p'$-group. For 
each $i$, set $S_i=N_S^{G_i}(T)=\{x\in S\,|\,c_x\in G_i\}$. Thus $S_i\nsg 
S$ and $\Aut_{S_i}(T)\in\sylp{G_i}$ for each $i$, $S_m=S$, and $S_0=TC_S(T)$. 

We claim that there are subsystems 
$\calf=\calf_m>\calf_{m-1}>\cdots>\calf_0=\cale$ in $\calf$, each normal 
in $\calf$ and of $p$-power index or of index prime to $p$ in the 
preceeding one, where $\calf_i$ is over $S_i$, and such that 
$\Aut_{\calf_i}(T)=G_i$ for each $i$. If this holds, then 
$(\calf_0)^\infty=\calf^\infty$ by Lemma \ref{unique-Finfty}, and 
$\calf_0$ satisfies all of the other conditions listed above.

Assume $\calf_i\nsg\calf$ has been constructed as claimed, with $i\ge1$ and 
$\calf_i\ge\cale$. If $G_i/G_{i-1}$ is a $p'$-group, then $S_{i-1}=S_i$. 
Let $\calf_{i-1}\nsg\calf_i$ be as in Proposition \ref{O^p'(F)}. In 
particular, $\calf_{i-1}$ has index prime to $p$ in $\calf_i$ and contains 
$\cale$. For each $\alpha\in\autf(S_i)$, $\alpha|_T\in\autf(T)=G_m$ 
normalizes $G_{i-1}=\Aut_{\calf_{i-1}}(T)$, and hence $\alpha$ normalizes 
$\calf_{i-1}$ by the uniqueness in Proposition \ref{O^p'(F)}. So 
$\calf_{i-1}\nsg\calf$ by Lemma \ref{F2<|F1<|F}.

If $G_i/G_{i-1}$ is a $p$-group, then $\hyp(\calf_i)\le S_{i-1}$ by 
Proposition \ref{O^p(F)}, so there is a unique normal subsystem 
$\calf_{i-1}\nsg\calf_i$ over $S_{i-1}$ of $p$-power index (see \cite[Theorem 
I.7.4]{AKO}), and $\calf_{i-1}\ge\cale$ by Proposition \ref{O^p(F)} again. 
For each $\alpha\in\autf(S_i)$, $\alpha(S_{i-1})=S_{i-1}$ since 
$S_{i-1}=N_S^{G_{i-1}}(T)$ and 
$\alpha|_T\in G_m$ normalizes $G_{i-1}$, so $\alpha$ normalizes 
$\calf_{i-1}$ by its uniqueness. Thus $\calf_{i-1}\nsg\calf$ by Lemma 
\ref{F2<|F1<|F}. 
\end{proof}

\bigskip

\section{Reductions to centric normal subsystems}

In order to get further results, we must also work with the linking system 
associated to a fusion system. When $G$ is a group and $P,Q\le G$, we set 
$T_G(P,Q)=\{g\in G\,|\,\9gP\le Q\}$ (the \emph{transporter set}). 
When $\calf$ is a saturated fusion system 
over a $p$-group $S$, a \emph{centric linking system} associated to $\calf$ 
is a category $\call$ with $\Ob(\call)=\calf^c$, the set of $\calf$-centric 
subgroups of $S$, together with a pair of functors
	\[ \calt_{\calf^c}(S) \Right4{\delta} \call \Right4{\pi} \calf \]
which satisfy certain axioms listed in \cite[Definition III.4.1]{AKO} and 
\cite[Definition 1.7]{BLO2}. Here, $\calt_{\calf^c}(S)$ is the 
\emph{transporter category} of $S$: the category with object set $\calf^c$, 
and with $\Mor_{\calt_{\calf^c}(S)}(P,Q)=T_S(P,Q)$ 
(where composition is given by multiplication in $S$). Also, 
$\delta$ is the identity on objects and injective on morphism 
sets, $\pi$ is the inclusion on objects and surjective on morphism sets, 
and $\pi\circ\delta$ sends $g\in T_S(P,Q)$ to 
$c_g\in\homf(P,Q)$. The motivating example is the category 
$\call_S^c(G)$, when $G$ is a finite group and $S\in\sylp{G}$, where 
	\begin{align*} 
	\Ob(\call_S^c(G)) &= \calf_S(G)^c = \{P\le S \,|\, 
	C_G(P)=Z(P)\times O_{p'}(C_G(P)) \} \\
	\Mor_{\call_S^c(G)}(P,Q) &=  T_G(P,Q)\big/ O_{p'}(C_G(P)) .
	\end{align*}
By comparison, note that $\Hom_{\calf_S(G)}(P,Q)\cong T_G(P,Q)/C_G(P)$. 

Since we are working with extensions of fusion and linking systems, we also 
need to handle their automorphism groups. Automorphisms of fusion systems 
are straightforward. When $\calf$ is a saturated fusion system over a 
$p$-group $S$, an automorphism $\alpha\in\Aut(S)$ is said to be 
\emph{fusion preserving} if it induces an automorphism of the category 
$\calf$, and we set 
	\begin{align*} 
	\Aut(S,\calf) &= \bigl\{\alpha\in\Aut(S)\,\big|\, 
	\textup{$\alpha$ is fusion preserving}\bigr\} \\
	\Out(S,\calf) &= \Aut(S,\calf)/\autf(S).
	\end{align*}

Let $\call$ be a centric linking system associated to $\calf$, and let 
$\delta$ be the functor described above. For each $P\in\calf^c=\Ob(\call)$, 
set $\iota_P=\delta_{P,S}(1)\in\Hom_\call(P,S)$ (the ``inclusion'' of $P$ 
in $S$ in the category $\call$), and set 
$[\![P]\!]=\delta_P(P)\le\Aut_\call(P)$. Define
	\begin{align*} 
	\Aut\typ^I(\call) &= \bigl\{ \beta\in\Aut(\call) \,\big|\, 
	\beta(\iota_P)=\iota_{\beta(P)},~
	\beta([\![P]\!])=[\![\beta(P)]\!],~ \forall\, 
	P\in\calf^c \bigr\} \\
	\Out\typ(\call) &= \Aut\typ^I(\call)\big/\gen{c_\gamma\,|\,
	\gamma\in\Aut_\call(S)}. 
	\end{align*}
There is a natural homomorphism $\mu_\call\:\Out\typ(\call) 
\Right2{}\Out(S,\calf)$, defined by restriction to 
$[\![S]\!]\le\Aut_\call(S)$. 
We refer to \cite[\S\,III.4.3]{AKO} or \cite[\S\,1.3]{AOV1} for 
more details on these definitions. 

When $\calf$ is a saturated fusion system over a $p$-group $S$, it is 
straightforward to define the centralizer fusion system $C_\calf(P)$ of a 
subgroup $P\le S$ (see \cite[Definition I.5.3]{AKO}): this is a fusion 
subsystem over $C_S(P)$ which is saturated if $P$ is fully centralized in 
$\calf$ (i.e., receptive). One can also define $C_\calf(\cale)$ when 
$\cale$ is a normal subsystem in $\calf$ (see \cite[Chapter 6]{A-gfit}), 
but this is much more complicated. For our purposes here, it will suffice 
to work with the following somewhat simpler definition. If $\cale\nsg\calf$ 
are saturated fusion systems over $T\nsg S$, then by \cite[6.7]{A-gfit}, 
there is a (unique) subgroup $C_S(\cale)\le C_S(T)$ with the property that 
for $P\le C_S(T)$, $P\le C_S(\cale)$ if and only if $\cale\le C_\calf(P)$. 

\begin{Prop} \label{S/T<Out(L)}
Let $\cale\nsg\calf$ be a pair of saturated fusion systems over the 
$p$-groups $T\nsg S$ such that $C_S(\cale)\le T$. Let $\call$ be a 
centric linking system associated to $\cale$. Then the natural 
homomorphism $S\too\Aut(T,\cale)$ lifts to a homomorphism 
$\til\omega\:S\too\Aut\typ^I(\call)$, and this factors through an injective 
homomorphism $\omega\:S/T\too\Out\typ(\call)$.
\end{Prop}

\begin{proof} Since the conclusion of the proposition involves only $S$ and 
$\cale$, we can assume that $\calf=S\cale$ as defined in \cite[Theorem 5, 
Chapter 8]{A-gfit}. In particular, $\cale\ge O^p(\calf)$. So by 
\cite[Proposition 1.31(a)]{AOV1}, there is a pair of linking systems 
$\call\nsg\call^*$ associated to $\cale\nsg\calf$, where 
$\Ob(\call)=\cale^c$, and $\Ob(\call^*)$ is the set of all $P\le S$ such 
that $P\cap T\in\Ob(\call)$. (This was shown in \cite{AOV1} only when 
$\cale=O^p(\calf)$, but the same argument applies in this situation.)

This inclusion $\call\nsg\call^*$ induces a natural homomorphism $\til\omega$ 
from $S$ to $\Aut\typ^I(\call)$, defined by conjugation in $\call^*$, and 
which factors through a homomorphism 
	\[ \omega\: S/T \Right4{} \Out\typ(\call) = 
	\Aut\typ^I(\call)\big/\{c_\gamma \,|\, 
	\gamma\in\Aut_\call(T) \}. \]
Assume $\omega$ is not injective, and let $x\in S\sminus T$ be such that 
$xT\in\Ker(\omega)$. Thus $\til\omega(x)=c_{\delta_T(x)}$ is conjugation by 
some element $\gamma\in\Aut_\call(T)$. Since $\til\omega(x)$ has $p$-power 
order, we can assume that $\gamma$ has $p$-power order, and hence 
$\gamma=\delta_T(y)$ for some $y\in T$. Upon replacing $x$ by $xy^{-1}$, we 
can arrange that $c_{\delta_T(x)}=\Id_\call$, and hence that $\delta_T(x)$ 
and its restrictions commute with all morphisms in $\call$. In particular, 
$x\in C_S(T)$.

Fix $P,Q\le T$ and $\psi\in\Iso_\call(P,Q)$. We just showed that 
$\delta_Q(x)\psi=\psi\delta_P(x)$. So by \cite[Proposition 4.e]{O-ext}, $\psi$ 
extends to an isomorphism $\4\psi\in\Iso_{\call^*}(P\gen{x},Q\gen{x})$. Set 
$y=\pi(\4\psi)(x)$, where $\pi(\4\psi)\in\homf(P\gen{x},Q\gen{x})$. By 
axiom (C) for a linking system \cite[Definition III.4.1]{AKO},
	\[ \4\psi\circ\delta_{P\gen{x}}(x) = 
	\delta_{Q\gen{x}}(y)\circ\4\psi. \]
But this is also equal to $\delta_{Q\gen{x}}(x)\circ\4\psi$ since 
extensions are unique in a linking system \cite[Proposition 4.e or 
4.f]{O-ext}, and so $\delta_{Q\gen{x}}(x)=\delta_{Q\gen{x}}(y)$. Since 
$\delta$ is injective by \cite[Proposition 4.c]{O-ext}, we have 
$x=y=\pi(\4\psi)(x)$.

Thus all isomorphisms in $\cale$ between objects in $\call$ extend to 
morphisms in $\calf$ which send $x$ to itself. Since $\Ob(\call)=\cale^c$, 
this statement holds for all morphisms in $\cale$ by Alperin's fusion 
theorem. So $\cale\le C_\calf(x)$, hence $x\in C_S(\cale)$, which 
contradicts the assumption that $C_S(\cale)\le T$. We conclude that 
$\omega$ is injective. 
\end{proof}

We saw in Theorem \ref{t:p-solv} the importance of getting control of the 
quotient group $TC_S(T)/T$, when $\cale\nsg\calf$ are saturated fusion 
systems over $T\nsg S$.

\begin{Cor} \label{S/T-abel}
Let $\cale\nsg\calf$ be saturated fusion systems over $T\nsg S$ such that 
$C_S(\cale)\le T$. Then $TC_S(T)/T$ is abelian, and $C_S(T)\le T$ if 
$p$ is odd. 
\end{Cor}

\begin{proof} Let $\call$ be a centric linking system associated to 
$\cale$. Consider the homomorphisms
	\[ S/T \Right5{\omega} \Out\typ(\call) \Right5{\mu=\mu_\call} 
	\Out(T,\cale)\le \Aut(T)/\Aut_\cale(T) \]
where $\omega$ is injective by Proposition \ref{S/T<Out(L)} and 
$\mu(\omega(xT))=[c_x]$ for $x\in S$. Thus 
$TC_S(T)/T=\Ker(\mu\circ\omega)$ injects into $\Ker(\mu)$. 
In particular, $C_S(T)\le T$ if $\mu$ is injective, and this always 
holds if $p$ is odd by \cite[Theorem C]{O-Ch} and \cite{GLynd}. Otherwise, 
$TC_S(T)/T$ is abelian since $\Ker(\mu)$ is abelian (see 
\cite[Proposition III.5.12]{AKO}). 
\end{proof}

We are now ready to prove our main result, which says that under 
appropriate conditions on $\cale\nsg\calf$, $\calf$ reduces down to $\cale$ 
in the sense which we have been studying.

\begin{Thm} \label{solv}
Let $\cale\nsg\calf$ be saturated fusion systems over $T\nsg S$ such that 
$C_S(\cale)\le T$. Assume either
\begin{enuma} 
\item $\autf(T)/\Aut_\cale(T)$ is $p$-solvable; or

\item $\Out(T,\cale)$ is $p$-solvable.

\end{enuma}
Then $\calf^\infty=\cale^\infty$. 
\end{Thm}

\begin{proof} Since 
	\[ \autf(T)/\Aut_\cale(T) \le \Aut(T,\cale)/\Aut_\cale(T) = 
	\Out(T,\cale), \]
$\autf(T)/\Aut_\cale(T)$ is $p$-solvable if $\Out(T,\cale)$ is 
$p$-solvable, and thus (b) implies (a). So from now on, we assume 
$\autf(T)/\Aut_\cale(T)$ is $p$-solvable. By Theorem \ref{t:p-solv}, it 
suffices to prove this when $S=TC_S(T)$ and $\autf(T)=\Aut_\cale(T)$. So by 
Corollary \ref{S/T-abel}, $S/T$ is abelian. 

Set $\calh=\{P\le S\,|\,P\ge C_S(T)\}$. If $P\in\calf^c$ and 
$P\notin\calh$, then $PC_S(T)>P$, so $N_{PC_S(T)}(P)>P$, and there is $x\in 
N_{C_S(T)}(P)\sminus P$. Then $c_x\in\Aut_S(P)$ induces the identity on 
$P\cap T$ and on $P/(P\cap T)$, so $c_x\in O_p(\autf(P))$. Also, 
$c_x\notin\Inn(P)$ since $x\notin P$ and $P\in\calf^c$, so 
$1\ne[c_x]\in\Out_S(P)\cap O_p(\outf(P))$.

Let $(\Psi,\5\Psi)\:\calf\too\calf/T$ be the morphism of fusion systems 
which sends $P$ to $PT/T$ and $\varphi\in\Mor(\calf)$ to the induced 
homomorphism between quotient groups \cite[Theorem II.5.12]{AKO}. Let 
$\calf^*\subseteq\calf^c$ be the full subcategory with objects in 
$\calh\cap\calf^c$. By definition, for each $P\in\calh$, $PT=TC_S(T)=S$ 
since $P\ge C_S(T)$. Define $\chi\:\Mor(\calf^*)\too\Aut_{\calf/T}(S/T)$ by 
sending $\varphi\in\Mor(\calf^*)$ to 
$\5\Psi(\varphi)\in\Aut_{\calf/T}(S/T)$. This clearly satisfies 
conditions (iii)--(v) in Lemma \ref{Ker(chi)}. Since we just checked 
condition (ii) in the lemma, and (i) is clear, we conclude that there is a 
normal fusion subsystem $\calf_0\nsg\calf$ containing $O^{p'}(\calf)$ such 
that $\Aut_{\calf_0}(S)=\Ker(\chi|_{\autf(S)})$ and 
hence $\Aut_{\calf_0/T}(S/T)=1$. 

Since $S/T$ is abelian and $\Aut_{\calf_0/T}(S/T)=1$, we have 
$\calf_0/T=\calf_{S/T}(S/T)$. So $\hyp(\calf_0)\le T$ by Lemma 
\ref{l:O^p(F)}(a). By \cite[Theorem I.7.4]{AKO}, there is a unique fusion 
subsystem $\calf_1\nsg\calf_0$ over $T$ of $p$-power index. Also, 
$\cale\nsg\calf_1$ by Proposition \ref{O^p(F)}, so $\cale$ has index prime 
to $p$ in $\calf_1$ by \cite[Lemma 1.26]{AOV1}. Thue 
$\calf^\infty=(\calf_0)^\infty=(\calf_1)^\infty=\cale^\infty$ by Lemma 
\ref{unique-Finfty}.
\end{proof}

To explain the motivation for Theorem \ref{solv}, we recall some 
definitions in \cite{AOV1}. When $\calf=\calf_S(G)$ and 
$\call=\call_S^c(G)$ for some finite group $G$ with $S\in\sylp{G}$, there 
is a natural homomorphism
	\[ \Out(G) \cong N_{\Aut(G)}(S)\big/\Aut_{N_G(S)}(G) 
	\Right5{\kappa_G} \Out\typ(\call), \]
defined by sending $\alpha\in N_{\Aut(G)}(S)$ to the automorphism of 
$\call$ induced by $\alpha$. See \cite[\S\,2.2]{AOV1} for more details. The 
fusion system $\calf$ is \emph{tamely realized} by $G$ if 
$\calf\cong\calf_S(G)$ and $\kappa_G$ is split surjective, and $\calf$ is 
tame if it is tamely realized by some finite group. 

Finally, a saturated fusion system $\calf$ is \emph{reduced} if 
$O_p(\calf)=1$ and $O^p(\calf)=\calf=O^{p'}(\calf)$. If $\calf$ is any 
fusion system, and $Q=O_p(\calf)$, set 
$\red(\calf)=(C_\calf(Q)/Z(Q))^\infty$: the \emph{reduction} of $\calf$. By 
\cite[Proposition 2.2]{AOV1}, $\red(\calf)$ is reduced for any saturated 
fusion system $\calf$. By \cite[Theorem A]{AOV1}, $\calf$ is tame if its 
reduction is tame.

\begin{Cor} \label{c:tame-solv}
Let $\cale\nsg\calf$ be saturated fusion systems over $T\nsg S$, where 
$C_S(\cale)\le T$, and where $\cale$ is simple and is tamely realized by a 
known simple group $K$. Then $\calf^\infty=\cale$, and $\calf$ is tamely 
realized by an extension of $K$. 
\end{Cor}

\begin{proof} By the Schreier conjecture (see \cite[Theorem 7.1.1]{GLS3}), 
$\Out(K)$ is solvable. Set $\call=\call_T^c(K)$. 
Then $\Out\typ(\call)$ is solvable since $\kappa_K$ is 
surjective, and $\Out(T,\cale)$ is solvable since $\mu_\call$ is 
surjective by \cite[Theorem C]{O-Ch} and \cite{GLynd}. Hence 
$\calf^\infty=\cale^\infty$ by Theorem \ref{solv}, 
$\red(\calf)=\cale^\infty=\cale$ since $\cale$ is simple, and so $\calf$ is 
tame by \cite[Theorem 2.20]{AOV1}. More precisely, $\calf$ is tamely 
realized by an extension of $K$ by successive applications of 
\cite[Proposition 2.16]{AOV1}, together with the existence of 
``compatible'' linking systems as made precise in the proof of 
\cite[Theorem 2.20]{AOV1}. 
\end{proof}

One can take this further by stating it in terms of the generalized Fitting 
subsystem $F^*(\calf)$ of a saturated fusion system $\calf$ 
\cite[9.9]{A-gfit}. 

\begin{Cor} \label{c:F*(F)<|F}
Let $\calf$ be a saturated fusion system. Assume that 
$F^*(\calf)=O_p(\calf)\cale$ (a central product), where $\cale\nsg\calf$ is 
quasisimple, and where $\cale/Z(\cale)$ is tamely realized by a known 
simple group $K$. Then $\red(\calf)\cong\cale/Z(\cale)$, and $\calf$ is 
tamely realized by a finite group $G$ such that $F^*(G)=O_p(G)H$, 
where $H$ is quasisimple and $H/Z(H)\cong K$.
\end{Cor}

\begin{proof} Set $Q=O_p(\calf)$. Then  $C_{F^*(\calf)}(Q)/Z(Q)\cong 
\cale/Z(\cale)$ is simple, and hence the pair $C_{F^*(\calf)}(Q)/Z(Q)\nsg 
C_\calf(Q)/Z(Q)$ satisfies the hypotheses of Corollary \ref{c:tame-solv}. 
So by that corollary, $\red(\calf)\cong\cale/Z(\cale)$, and 
$C_\calf(Q)/Z(Q)$ is tamely realized by an extension of $K$. Together with 
\cite[Theorem 2.20]{AOV1} (and its proof), this implies that $\calf$ itself 
is tamely realized by a finite group $G$ of the form described above.
\end{proof}

%%%%%%%%%%%%%%%%%%%%%%%%%%%%%%%%%%%%%%%

\end{document}